\documentclass{elsarticle}

\usepackage{latexsym,bm,amssymb,amsfonts,amsbsy,amsmath,graphics,color}

\newtheorem{theo}{Theorem}
\newtheorem{lem}{Lemma}

\newtheorem{rem}{Remark}

\newtheorem{defn}{Definition}
\def\RR{\mathbb R}
\def\CC{\mathbb C}
\def\pmatrix{ \left( \begin{array} }
\def\endpmatrix{ \end{array} \right) }

\def\bfy{{\bf y}}

\def\bfo{{\bf 0}}

\def\no{\noindent}
\def\diag{{\rm diag}}
\def\proof{\underline{Proof}\quad}
\def\QED{~\mbox{$\Box$}}

\def\phi{\varphi}

\def\P{{\cal P}}

\def\I{{\cal I}}
\def\P{{\cal P}}

\def\sigmd{{\dot\sigma}}
\def\O{\Omega}

\def\dd{\mathrm{d}}

\begin{document}

\title{The Lack of Continuity and the Role of Infinite and Infinitesimal in
Numerical Methods for ODEs: the Case of Symplecticity
\tnoteref{t1}}

\tnotetext[t1]{Work developed within the project ``Numerical methods and
software for differential equations''.}

\author[lb]{Luigi Brugnano}\ead{luigi.brugnano@unifi.it}
\author[fi]{Felice Iavernaro}\ead{felix@dm.uniba.it}
\author[dt]{Donato Trigiante}\ead{trigiant@unifi.it}

\address[lb]{Dipartimento di Matematica,
Universit\`a di Firenze, Viale Morgagni 67/A, 50134 Firenze
(Italy).}
\address[fi]{Dipartimento di Matematica, Universit\`a di
Bari, Via Orabona  4,  70125 Bari (Italy).}
\address[dt]{Dipartimento di Energetica, Universit\`a di
Firenze, Via Lombroso 6/17, 50134 Firenze (Italy).}

\begin{abstract}
When numerically integrating canonical Hamiltonian systems, the long-term
conservation of some of its invariants, among which the Hamiltonian function
itself, assumes a central role. The classical approach to this problem has led
to the definition of symplectic methods, among which we mention Gauss-Legendre
collocation formulae. Indeed, in the continuous setting,
energy conservation is derived from symplecticity via an {\em infinite}
number of {\em
infinitesimal contact transformations}. However, this infinite process cannot be
directly transferred to the discrete setting. By following a different
approach, in this paper we describe a sequence of methods, sharing the same
essential spectrum (and, then, the same essential properties), which are energy
preserving starting from a certain element of the sequence on,
i.e., after a finite number of steps.\end{abstract}


\maketitle

\section{Introduction}\label{intro} In order to make easier the
following considerations, it is necessary to stress the
differences between a continuous problem, for example
(\ref{hamilode}), and a discrete problem obtained by applying to
it a whatever numerical method.  The main difference between them, very often
underrated by many authors, is  the lack of continuity in time. The latter
 does not affect many aspects such as, e.g., the study
of the qualitative behavior of solutions around asymptotically
stable critical points (stability analysis). In fact, the
respective theories, with minor changes, are very similar (see,
e.g., \cite{LT}). As a consequence, many tools, already devised in
the continuous analysis, can be transferred to the discrete
analysis  almost unchanged. This is the case, for example, of the
linearization around asymptotically stable critical points, which
has been extensively used  in the numerical analysis of methods
for differential problems (linear stability analysis). There are,
however, other aspects for which continuity plays an essential
role. For example, two solutions of the continuous problem need to
stay away from each other while, in the discrete case, they may
interlace (without having common points, of course). This fact has
many mathematical and even physical implications (see, e.g.,
\cite{IT07}). In this paper we shall deal with another case in
which continuity plays an essential role. It regards the role of
symplecticity which is central in discussing energy conservation
in continuous Hamiltonian problems, while it is less crucial in
the energy conservation of discrete problems. This depends on the
interplay between infinitesimal contact transformations and the
need of infinite processes (number of iterations) which cannot be
operatively used in Numerical Analysis. This question has been
already discussed in previous papers (see, e.g., \cite{BITUOVO}).
Here, after a rapid introduction to the subject, we shall focus on
a particular aspect, although very important, which concerns a
property of the numerical methods, in the general class of
collocation methods, which comes out by no more requiring
symplecticity but still providing conservation of the Hamiltonian
functions on a subset of points of the mesh. Clearly, this permits
to avoid the drift of energy experienced when using many numerical
methods proposed in the recent literature.

With this premise, the structure of the paper is the following: in
Section~\ref{chp} we recall the basic facts about canonical
Hamiltonian problems and the approaches used for their numerical
solution; one of them, resulting in the recently introduced class
of {\em Hamiltonian Boundary Value Methods (HBVMs)}, is then
sketched in Section~\ref{hbvms}; in Section~\ref{iso} we state the
main result of this paper, concerning the {\em isospectral
property} of such methods; in Section~\ref{coll} such property is
further generalized to study the existing connections between
HBVMs and Runge-Kutta collocation methods; a few concluding
remarks are finally given in Section~\ref{fine}.

\section{Canonical Hamiltonian problems}\label{chp}

Canonical Hamiltonian problems are in the form
\begin{equation}\label{hamilode}\dot y = J\nabla H(y),  \qquad y(t_0) =
y_0\in\RR^{2m},
\end{equation}

\no where $J$ is a skew-symmetric constant matrix, the Hamiltonian $H(y)$ is
assumed to be sufficiently differentiable, and the state vector
splits into two blocks, $y=(q^T,p^T)^T$,
$q,p\in\RR^m$ where, for mechanical systems, $q$ denote the positions
and $p$ the (generalized) momenta.
Such problems are of great interest in many fields of application, ranging from
the macro-scale of celestial mechanics, to the micro-scale of molecular
dynamics. They have been deeply studied, from the point of view of the
mathematical analysis, since two centuries. Their numerical solution is a more
recent field of investigation, where the main difficulty in dealing
with them numerically  stems from the fact that the meaningful
isolated critical points of such systems are only marginally stable:
neighboring solution curves do not eventually approach the
equilibrium point either in future or in past times.
This implies that the geometry around them critically depends on
perturbations of the linear part. Consequently, the use of a
linear test equation, which essentially captures the geometry of
the linear part, whose utility has been enormous in settling the
dissipative case, cannot be of any utility in the present case.

It is then natural to look for other properties of Hamiltonian
systems that can be imposed on the discrete methods in order to
make them effective. The first property which comes to mind is the
symplecticity  of the flow $\varphi_t := y_0 \mapsto y(t)$
associated with \eqref{hamilode}. This property can be described either
in geometric form (invariance of areas, volumes, etc.) or in
analytical form: $$\left(\frac{\partial \varphi_t}{\partial
y_0}\right)^T J \left(\frac{\partial \varphi_t}{\partial
y_0}\right)=J.$$ In one way or the other, it essentially consists
in moving {\em infinitesimally} on the trajectories representing the
solutions. Infinitesimally means retaining only the linear part of
the {\em infinitesimal time displacement} $\delta t$. It can be shown
that this produces  new values of the variables $q+\delta q,$
$p+\delta p$ which leave unchanged the value of the Hamiltonian
$H(q+\delta q,p+\delta p)=H(q,p)$ ({\em Infinitesimal Contact
Transformation} (ICT), see \cite[p.\,386]{Gold}).
Consequently, since the composition of  such
{\em infinitesimal} transformations maintains the invariance, so
does an {\em infinite number} of them.

 It is  not
surprising that the first numerical attempts to design
conservative methods  have tried to transfer similar arguments to
discrete methods, i.e. to design  symplectic integrators
\cite{Ruth,Feng} (see also the monographs \cite{SC,LeRe,HLW}  for
more details on the subject).  A backward error analysis has shown
that symplecticity seems somehow to improve  the long-time
behavior properties of the numerical solutions. Indeed, for a
symplectic method of order $r$, implemented with a constant
stepsize $h$, the following estimation reveals how the numerical
solution $y_n$ may depart from the manifold $H(y)=H(y_0)$ of the
phase space, which contains the continuous solution itself:
\begin{equation}
\label{backward} H(y_n)-H(y_0) = O(nh\,e^{-\frac{h_0}{2h}} h^r),
\end{equation}
where {\em $h_0>0$ is sufficiently small} and $h\le h_0$. Relation
\eqref{backward} implies that a linear drift of the energy with
respect to time $t=nh$ may arise. However, due to the presence of
the exponential, such a drift will not appear as far as $nh\le
e^{\frac{h_0}{2h}}$: this circumstance is often referred to by
stating that  {\em symplectic methods conserve the energy function
on exponentially long time intervals} (see, for example,
\cite[Theorem 8.1, p.\,367]{HLW}). This is clearly a surrogate of the
definition of stability in that the ``good behaviour'' of the
numerical solution is not extended on {\em infinite} time intervals.

As matter of fact, since symplecticity requires an \emph{infinite}
sequence of \emph{infinitesimal contact transformations} it cannot
be transferred \emph{``sic et simpliciter"} to the discrete
methods, simply because infinite processes are not permitted in
Numerical Analysis. A more efficient approach would require to
design methods which avoid the necessity of using infinite
processes while preserving the constant of motion, i.e., yielding a numerical
solution belonging to the manifold
$H(y)=H(y_0)$.
In this paper we consider
a recently introduced class of methods of any high order that
provide energy conservation. More specifically, with any given
order, one can associate an infinite sequences of methods, differing
from each other for the number of internal mesh points that cover
the same time window, say  $[t_i, t_i+h]$: the more points
we include, the better the conservation properties of the method.
However, we show that, at least for polynomial Hamiltonian
functions, such a  process of increasing the number of internal
points is not infinite. In fact we show that there exists a finite
value of new added points starting from which the method become
conservative, whatever is the stepsize $h$ used.

%In this paper, although we introduce new points in
%the interval of integration, which somehow is equivalent to use a
%smaller stepsize, we define methods for which this process is not
%infinite, since we show that there exists a finite value of new
%added  points for which the methods start to be conservative, at
%least for polynomial Hamiltonians.

The evolution of the approaches to the problem, i.e. to get
efficient energy-conserving methods, has been slow. As a matter of
fact, the first unsuccessful attempts to derive energy-preserving
Runge-Kutta methods culminated in the wrong general feeling that
such methods could not even exist \cite{IZ}. One of the first
successful attempts to solve the problem, outside the class of
Runge-Kutta methods, is represented by {\em discrete gradient
methods} (see \cite{McLQR} and references therein) which are
second order accurate. Purely algebraic approaches have been also
introduced (see, e.g., \cite{CFM}), without presenting any energy
preserving method. A further approach was considered in
\cite{QMcL}, where the {\em averaged vector field method} was
proposed and shown to conserve the energy function, although it is
only second order accurate. As was recently outlined in
\cite{M2AN}, approximating the integral appearing in such method
by means of a quadrature formula (based upon polynomial
interpolation) yields a family of second order Runge-Kutta
methods. These latter methods represent an instance of
energy-preserving Runge-Kutta methods for polynomial Hamiltonian
problems: their first appearance may be found in \cite{IP1}, under
the name of {\em $s$-stage trapezoidal methods}. Additional
examples of fourth and sixth-order conservative Runge-Kutta
methods (for polynomial Hamiltonians of suitable degree) were
presented in \cite{IP2,IT}. All such energy-conserving methods
have been derived  by means of the new concept of {\em discrete
line integral}.

The evolution of this
idea eventually led to the definition of {\em Hamiltonian Boundary
Value Methods (HBVMs)} \cite{BIT2,BIT,BIT1}, which is a wide class of methods
able to preserve, for the discrete solution, polynomial Hamiltonians of
arbitrarily high degree (and then, a {\em practical} conservation
of any sufficiently differentiable Hamiltonian). In more details,
in \cite{BIT} HBVMs defined at Lobatto nodes have been analysed,
whereas in \cite{BIT1} HBVMs defined at Gauss-Legendre abscissae
have been considered. In the last reference, it has been actually
shown that both formulae are essentially equivalent to each other,
since the order and stability properties of the methods turn out
to be independent of the abscissae distribution, and both methods
are equivalent, when the number of the so called {\em silent
stages} tends to {\em infinity}. In this paper this conclusion is
further supported, since we prove that HBVMs, when cast as
Runge-Kutta methods, are such that  the corresponding matrix of
the tableau has the nonzero eigenvalues coincident with those of
the corresponding Gauss-Legendre formula ({\em isospectral
property} of HBVMs). This property will be also used to further
analyse the existing connections between HBVMs and Runge-Kutta
collocation methods.

\section{Hamiltonian Boundary Value Methods}\label{hbvms}

The arguments in this section are worked out starting from those
used in \cite{BIT,BIT1} to introduce and analyse HBVMs.
Starting from the canonical Hamiltonian problems
(\ref{hamilode}), the key formula which HBVMs rely on, is the {\em line
integral} and the related property of conservative vector fields:
\begin{equation}\label{Hy}
H(y_1) - H(y_0) = h\int_0^1 \sigmd(t_0+\tau h)^T\nabla
H(\sigma(t_0+\tau h))\dd\tau,
\end{equation}

\no for any $y_1 \in \RR^{2m}$, where $\sigma$ is any smooth
function such that
\begin{equation}
\label{sigma}\sigma(t_0) = y_0, \qquad\sigma(t_0+h) = y_1.
\end{equation}

\no Here we consider the case where $\sigma(t)$ is a polynomial of
degree $s$, yielding an approximation to the true solution $y(t)$
in the time interval $[t_0,t_0+h]$. The numerical approximation
for the subsequent time-step, $y_1$, is then defined by
(\ref{sigma}).
After introducing a set of $s$ distinct abscissae,
\begin{equation}\label{ci}0<c_{1},\ldots ,c_{s}\le1,\end{equation}

\no we set
\begin{equation}\label{Yi}Y_i=\sigma(t_0+c_i h), \qquad
i=1,\dots,s,\end{equation}

\no so that $\sigma(t)$ may be thought of as an interpolation
polynomial, interpolating the {\em fundamental stages} $Y_i$,
$i=1,\dots,s$, at the abscissae (\ref{ci}). We observe that, due
to (\ref{sigma}), $\sigma(t)$ also interpolates the initial
condition $y_0$.

\begin{rem}\label{c0} Sometimes, the interpolation at $t_0$ is
explicitly  required. In such a case, the extra abscissa $c_0=0$ is
formally added to (\ref{ci}). This is the case, for example, of a
Lobatto distribution of the abscissae \cite{BIT}.\end{rem}

Let us consider the following expansions of $\dot \sigma(t)$ and
$\sigma(t)$ for $t\in [t_0,t_0+h]$:
\begin{equation}
\label{expan} \dot \sigma(t_0+\tau h) = \sum_{j=1}^{s} \gamma_j
P_j(\tau), \qquad \sigma(t_0+\tau h) = y_0 + h\sum_{j=1}^{s}
\gamma_j \int_{0}^\tau P_j(x)\,\dd x,
\end{equation}

\no where $\{P_j(t)\}$ is a suitable basis of the vector space of
polynomials of degree at most $s-1$ and  the  (vector)
coefficients $\{\gamma_j\}$ are to be determined.  We shall
consider an orthonormal polynomial basis on the interval $[0,1]$
(though, in principle, different bases could be also considered
\cite{BIT1,IP1,IT,IP2}):
\begin{equation}\label{orto}\int_0^1 P_i(t)P_j(t)\dd t = \delta_{ij}, \qquad
i,j=1,\dots,s,\end{equation}

\no where $\delta_{ij}$ is the Kronecker symbol, and $P_i(t)$ has
degree $i-1$. Such a basis can be readily obtained as
\begin{equation}\label{orto1}P_i(t) = \sqrt{2i-1}\,\hat P_{i-1}(t), \qquad
i=1,\dots,s,\end{equation}

\no with $\hat P_{i-1}(t)$ the shifted Legendre polynomial, of
degree $i-1$, on the interval $[0,1]$ (see, e.g., \cite{AS}).
We shall also assume that $H(y)$ is a polynomial, which implies
that the integrand in \eqref{Hy} is also a polynomial so that the
line integral can be exactly computed by means of a suitable
quadrature formula. It is easy to observe that in general, due to
the high degree of the integrand function,  such quadrature
formula cannot be solely based upon the available abscissae
$\{c_i\}$: one needs to introduce an additional set of abscissae
$\{\hat c_1, \dots,\hat c_r\}$, distinct from the nodes $\{c_i\}$,
in order to make the quadrature formula exact:
\begin{eqnarray} \label{discr_lin}
\displaystyle \lefteqn{\int_0^1 \sigmd(t_0+\tau h)^T\nabla
H(\sigma(t_0+\tau h))\mathrm{d}\tau   =}\\ && \sum_{i=1}^s \beta_i
\sigmd(t_0+c_i h)^T\nabla H(\sigma(t_0+c_i h)) + \sum_{i=1}^r \hat
\beta_i \sigmd(t_0+\hat c_i h)^T\nabla H(\sigma(t_0+\hat c_i h)),
\nonumber
\end{eqnarray}

\no where $\beta_i$, $i=1,\dots,s$, and $\hat \beta_i$,
$i=1,\dots,r$, denote the weights of the quadrature formula
defined at the abscissae $\{c_i\}\cup\{\hat c_i\}$.
Then, according to \cite{BIT,BIT1}, we give the following
definition.

\begin{defn}\label{defhbvmks}
The method defined by the polynomial $\sigma(t)$, determined by
substituting the quantities in \eqref{expan} into the right-hand
side of \eqref{discr_lin}, and by choosing the unknown coefficient
$\{\gamma_j\}$ in order that the resulting expression vanishes, is
called {\em Hamiltonian Boundary Value Method with $k$ steps and
degree $s$}, in short {\em HBVM($k$,$s$)}, where
$k=s+r$.\end{defn}

According to \cite{IT2}, the right-hand side of \eqref{discr_lin} is
called \textit{discrete line integral} associated with the map
defined by the HBVM($k$,$s$) method, while the vectors
\begin{equation}\label{hYi}
\hat Y_i \equiv \sigma(t_0+\hat c_i h), \qquad i=1,\dots,r,
\end{equation}

\no are called \textit{silent stages}: they just serve to increase,
as much as one likes, the degree of precision of the quadrature
formula, but they are not to be regarded as unknowns since, from
\eqref{expan} and (\ref{hYi}), they can be expressed in terms of
linear combinations of the fundamental stages (\ref{Yi}).

Because of the equality \eqref{discr_lin}, we can apply the
procedure described in Definition \ref{defhbvmks} directly to the
original line integral appearing in the left-hand side. With this
premise, by considering the first expansion in \eqref{expan},  the
conservation property reads
$$\sum_{j=1}^{s} \gamma_j^T \int_0^1  P_j(\tau) \nabla
H(\sigma(t_0+\tau h))\dd\tau=0,$$

\no which, as is easily checked,  is satisfied if we
impose the following set of orthogonality conditions:
\begin{equation}
\label{orth} \gamma_j = \int_0^1  P_j(\tau) J \nabla
H(\sigma(t_0+\tau h))\dd\tau, \qquad j=1,\dots,s.
\end{equation}

\no Then, from the second relation of \eqref{expan} we obtain, by
introducing the operator
\begin{eqnarray}\label{Lf}\lefteqn{L(f;h)\sigma(t_0+ch) =}\\
\nonumber && \sigma(t_0)+h\sum_{j=1}^s \int_0^c P_j(x) \dd x \,
\int_0^1 P_j(\tau)f(\sigma(t_0+\tau h))\dd\tau,\qquad
c\in[0,1],\end{eqnarray}

\no that $\sigma$ is the eigenfunction of $L(J\nabla H;h)$
relative to the eigenvalue $\lambda=1$:
\begin{equation}\label{L}\sigma = L(J\nabla H;h)\sigma.\end{equation}

\no According to \cite{BIT1}, (\ref{L}) is called the {\em Master
Functional Equation} defining $\sigma$: it characterizes HBVM$(k,s)$ methods,
for all $k\ge s$. Indeed, such methods are uniquely defined by the
polynomial $\sigma$, of degree $s$, the number of steps $k$ being
only required to obtain the exact quadrature formula
(\ref{discr_lin}).

To practically compute $\sigma$, we set (see (\ref{Yi}) and
(\ref{expan}))
\begin{equation}
\label{y_i} Y_i=  \sigma(t_0+c_i h) = y_0+ h\sum_{j=1}^{s} a_{ij}
\gamma_j, \qquad i=1,\dots,s,
\end{equation}

\no where
$$a_{ij}=\int_{0}^{c_i} P_j(x) \mathrm{d}x, \qquad
i,j=1,\dots,s.$$%

\no Inserting \eqref{orth} into \eqref{y_i} yields the final
formulae which define the HBVMs class based upon the orthonormal
basis $\{P_j\}$:
\begin{equation}
\label{hbvm_int} Y_i=y_0+h  \sum_{j=1}^s a_{ij}\int_0^1
P_j(\tau)  J \nabla H(\sigma(t_0+\tau h))\,\dd\tau, \qquad
i=1,\dots,s.
\end{equation}

For sake of completeness, we report the nonlinear system
associated with the HBVM$(k,s)$ method, in terms of the
fundamental stages $\{Y_i\}$ and the silent stages $\{\hat Y_i\}$
(see (\ref{hYi})), by using the notation
\begin{equation}\label{fy}
f(y) = J \nabla H(y).
\end{equation}

\no It  represents the discrete counterpart of \eqref{hbvm_int},
which may be directly retrieved by evaluating, for example, the
integrals in \eqref{hbvm_int} by means of the (exact) quadrature
formula introduced in \eqref{discr_lin}:
\begin{eqnarray}\label{hbvm_sys}
\lefteqn{ Y_i =}\\
&=& y_0+h\sum_{j=1}^s a_{ij}\left( \sum_{l=1}^s \beta_l
P_j(c_l)f(Y_l) + \sum_{l=1}^r\hat \beta_l P_j(\hat c_l) f(\widehat
Y_l) \right),\quad i=1,\dots,s.\nonumber
\end{eqnarray}

\no From the above discussion it is clear that, in the
non-polynomial case, supposing to choose the abscissae $\{\hat
c_i\}$ so that the sums in (\ref{hbvm_sys}) converge to an
integral as $r\equiv k-s$ tends to {\em infinity}, the resulting
formula is \eqref{hbvm_int}, which has been named {\em
$\infty$-HBVM of degree $s$} or {\em HBVM$(\infty,s)$} in
\cite{BIT1}. This implies that HBVMs may be as well applied in the
non-polynomial case since, in finite precision arithmetic, HBVMs
are undistinguishable from their limit formulae \eqref{hbvm_int},
when a sufficient number of silent stages is introduced, so that a
{\em practical} energy conservation is obtained, for $k$ large
enough \cite{BIT,BIT1,IP1,IT}. On the other hand, we emphasize
that, in the non-polynomial case, \eqref{hbvm_int} becomes an {\em
operative method} only after that a suitable strategy to
approximate the integrals appearing in it is taken into account.
In the present case, if one discretizes the {\em Master Functional
Equation} (\ref{Lf})--(\ref{L}), HBVM$(k,s)$ are then obtained,
essentially by extending the discrete problem (\ref{hbvm_sys})
also to the silent stages (\ref{hYi}). In more details, by using
(\ref{fy}) and introducing the following notation:
\begin{eqnarray*}
\{\tau_i\} = \{c_i\} \cup \{\hat{c}_i\}, &&
\{\omega_i\}=\{\beta_i\}\cup\{\hat\beta_i\},\\[2mm]
y_i = \sigma(t_0+\tau_ih), && f_i = f(\sigma(t_0+\tau_ih)), \qquad
i=1,\dots,k,
\end{eqnarray*}
the discrete problem defining the HBVM$(k,s)$ method becomes,
\begin{equation}\label{hbvmks}
y_i = y_0 + h\sum_{j=1}^s \int_0^{\tau_i} P_j(x)\dd x
\sum_{\ell=1}^k \omega_\ell P_j(\tau_\ell)f_\ell, \qquad
i=1,\dots,k.
\end{equation}

\no By defining the vectors ~$\bfy = (y_1^T,\dots,y_k^T)^T$~ and
~$e=(1,\dots,1)^T\in\RR^k$,~  and the matrices
\begin{equation}\label{OIP}\O=\diag(\omega_1,\dots,\omega_k), \qquad
\I_s,~\P_s\in\RR^{k\times s},\end{equation} whose $(i,j)$th entry
are given by
\begin{equation}\label{IDPO}
(\I_s)_{ij} = \int_0^{\tau_i} P_j(x)\mathrm{d}x, \qquad
(\P_s)_{ij}=P_j(\tau_i), \end{equation}

\no we can cast the set of equations (\ref{hbvmks}) in vector form
as $$\bfy = e\otimes y_0 + h(\I_s
\P_s^T\O)\otimes I_{2m}\, f(\bfy),$$

\no with an obvious meaning of $f(\bfy)$. Consequently, the method
can be regarded as a Runge-Kutta method with the following Butcher
tableau:
\begin{equation}\label{rk}
\begin{array}{c|c}\begin{array}{c} \tau_1\\ \vdots\\ \tau_k\end{array} & \I_s \P_s^T\O\\
 \hline                    &\omega_1\, \dots~ \omega_k
                    \end{array}\end{equation}
In particular, when a Gauss distribution of the abscissae
$\{\tau_1,\dots,\tau_k\}$ is considered, it can be proved that the
resulting HBVM$(k,s)$ method \cite{BIT1} (see also \cite{BIT,HBVMsHome}):

\begin{itemize}

\item has order $2s$ for all $k\ge s$;

\item is symmetric and perfectly $A$-stable (i.e., its
stability region coincides with the left-half complex plane,
$\CC^-$ \cite{BT});

\item reduces to the Gauss-Legendre method of order $2s$, when
$k=s$;

\item exactly preserves polynomial Hamiltonian functions of degree $\nu$,
provided that \begin{equation}\label{knu}k\ge \frac{\nu
s}2.\end{equation}

\end{itemize}

\section{The Isospectral Property}\label{iso}

We are now going to prove a further additional result, related to
the matrix appearing in the Butcher tableau (\ref{rk}),
 i.e., the matrix
\begin{equation}\label{AMAT}A = \I_s \P_s^T\O\in\RR^{k\times
k}, \qquad k\ge s,\end{equation}

\no whose rank is $s$. Consequently it has a $(k-s)$-fold zero
eigenvalue. In this section, we are going to discuss its {\em essential
spectrum}, i.e., the location of the remaining $s$ nonzero eigenvalues of that
matrix. Before that, we state a couple of preliminary results: their proofs
follow, respectively, from \cite[Theorem\,5.6, p.\,83]{HW} and from the
properties of shifted Legendre polynomials (see, e.g., \cite{AS} or the Appendix
in \cite{BIT}).

\begin{lem}\label{gauss} The eigenvalues of the matrix
\begin{equation}\label{Xs}
X_s = \pmatrix{cccc}
\frac{1}2 & -\xi_1 &&\\
\xi_1     &0      &\ddots&\\
          &\ddots &\ddots    &-\xi_{s-1}\\
          &       &\xi_{s-1} &0\\
\endpmatrix, \end{equation} with
\begin{equation}\label{xij}\xi_j=\frac{1}{2\sqrt{(2j+1)(2j-1)}}, \qquad
j\ge1,\end{equation} coincide with those of the matrix in the
Butcher tableau of the Gauss-Legendre method of order
$2s$.\end{lem}

\begin{lem}\label{intleg} With reference to the matrices in
(\ref{OIP})--(\ref{IDPO}), one has
$$\I_s = \P_{s+1}\hat{X}_s,$$
where
$$\hat{X}_s = \pmatrix{cccc}
\frac{1}2 & -\xi_1 &&\\
\xi_1     &0      &\ddots&\\
          &\ddots &\ddots    &-\xi_{s-1}\\
          &       &\xi_{s-1} &0\\
\hline &&&\xi_s\endpmatrix,$$

\no with the $\xi_j$ defined by (\ref{xij}). \end{lem}

\no The following result then holds true.

\begin{theo}[Isospectral Property of HBVMs]\label{mainres}
For all $k\ge s$ and for any choice of the abscissae $\{\tau_i\}$
such that the quadrature defined by the weights $\{\omega_i\}$ is
exact for polynomials of degree $2s-1$, the nonzero eigenvalues of
the matrix $A$ in (\ref{AMAT}) coincide with those of matrix
(\ref{Xs}), characterizing the Gauss-Legendre method of order $2s$.
\end{theo}
\begin{proof}
For $k=s$, the abscissae $\{\tau_i\}$ have to be the $s$
Gauss-Legendre nodes, so that HBVM$(s,s)$ reduces to the Gauss
Legendre method of order $2s$, as already outlined at the end of
Section~\ref{hbvms}.
When $k>s$, from the orthonormality of the basis, see
(\ref{orto}), and considering that the quadrature with weights
$\{\omega_i\}$ is exact for polynomials of degree (at least)
$2s-1$, one obtains that (see (\ref{OIP})--(\ref{IDPO})) for all
~$i=1,\dots,s$~ and ~$j=1,\dots,s+1$,
$$\left(\P_s^T\O\P_{s+1}\right)_{ij} = \sum_{\ell=1}^k \omega_\ell
P_i(t_\ell)P_j(t_\ell)=\int_0^1 P_i(t)P_j(t)\dd t = \delta_{ij},$$
and, therefore,
$$\P_s^T\O\P_{s+1} = \left( I_s ~ \bfo\right).$$
By taking into account the result of Lemma~\ref{intleg}, one
then obtains:
\begin{eqnarray}\nonumber
A\P_{s+1} &=& \I_s \P_s^T\O\P_{s+1} = \I_s \left(I_s~\bfo\right)
=\P_{s+1}
\hat{X}_s \left(I_s~\bfo\right) = \P_{s+1}\left(\hat{X}_s~\bfo\right)\\
 &=& \P_{s+1}
\pmatrix{cccc|c}
\frac{1}2 & -\xi_1 && &0\\
\xi_1     &0      &\ddots& &\vdots\\
          &\ddots &\ddots    &-\xi_{s-1}&\vdots\\
          &       &\xi_{s-1} &0&0\\
\hline &&&\xi_s&0\endpmatrix ~\equiv~ \P_{s+1}\widetilde X_s,
\label{tXs}
\end{eqnarray}

\no with the $\{\xi_j\}$ defined according to (\ref{xij}).
Consequently, one obtains that the columns of $\P_{s+1}$
constitute a basis of an invariant (right) subspace of matrix $A$,
so that the eigenvalues of $\widetilde X_s$ are eigenvalues of
$A$. In more detail, the eigenvalues of $\widetilde X_s$ are those
of $X_s$ (see (\ref{Xs})) and the zero eigenvalue. Then, also in
this case, the nonzero eigenvalues of $A$ coincide with those of
$X_s$, i.e., with the eigenvalues of the matrix defining the
Gauss-Legendre method of order $2s$.\QED\end{proof}

It turns out that such methods, in the form here presented (i.e., having chosen
the polynomial basis (\ref{orto1})), can be regarded as a generalization of
Gauss methods, in the sense that, they share the same nonzero spectrum, for all
$k\ge s$. In the limit $k\rightarrow\infty$, the same essential spectrum is
retained by the limit operator (see (\ref{L})). In the case of a polynomial
Hamiltonian, such sequence of methods  starts to be energy-preserving for a
finite value of $k$. Moreover, even though  for a general Hamiltonian the
method becomes energy-preserving in the limit $k\rightarrow\infty$,
nevertheless, when using finite precision arithmetic, the limit is practically
obtained for a finite value of $k$, namely as soon as full machine precision
accuracy is achieved.

\section{HBVMs and Runge-Kutta collocation methods}\label{coll}

By using the previous results and notations, we now further
elucidate the existing connections between HBVMs and Runge-Kutta
collocation methods.
Our starting point is a generic collocation method with $k$
stages, defined by the tableau
\begin{equation}
\label{collocation_rk}
\begin{array}{c|c}\begin{array}{c} \tau_1\\ \vdots\\ \tau_k\end{array} &  \mathcal A \\
 \hline                    &\omega_1\, \ldots  ~ \omega_k
\end{array}
\end{equation}

\no where, for $i,j=1,\dots,k$: $$\mathcal A=
(\alpha_{ij})\equiv\left(\int_0^{\tau_i} \ell_j(x) \mathrm{d}x
\right), \qquad \omega_j=\int_0^{1} \ell_j(x) \mathrm{d}x,$$

\no $\ell_j(\tau)$ being the $j$th Lagrange polynomial of degree
$k-1$ defined on the set of abscissae $\{\tau_i\}$. Moreover,
given a positive integer $s\le k$, and considering the
matrices defined in (\ref{OIP})--(\ref{IDPO}), we consider the matrix
$$\P_s \P_s^T\Omega\in\RR^{k\times k}$$

\no with projects into the $s$-dimensional subspace spanned by the columns of
$\P_s$. The class of Runge-Kutta methods we are interested in,
is then defined by the tableau
\begin{equation}
\label{hbvm_rk1}
\begin{array}{c|c}\begin{array}{c} \tau_1\\ \vdots\\ \tau_k\end{array} &
A \equiv \mathcal A (\P_s \P_s^T \Omega)\\
 \hline       &\omega_1\, \ldots \ldots ~ \omega_k
\end{array}
\end{equation}

\no We note that the Butcher array $A$ has rank which cannot exceed
$s$, because it is defined by {\em filtering} $\mathcal A$ by the
rank $s$ matrix $\P_s \P_s^T \Omega$.
The following result then holds true, which clarifies the existing
connections between classical Runge-Kutta collocation methods and
HBVMs.

\begin{theo}\label{collhbvm} Provided that the quadrature formula defined by the
weights $\{\omega_i\}$ is exact for polynomials at least $2s-1$
(i.e., the Runge-Kutta method defined by the tableau
(\ref{hbvm_rk1}) satisfies the usual simplifying assumption
$B(2s)$), then the tableau (\ref{hbvm_rk1}) defines a HBVM$(k,s)$
method based at the abscissae $\{\tau_i\}$.
\end{theo}

\proof Let us expand the basis $\{P_1(\tau),\dots,P_s(\tau)\}$
along the Lagrange basis $\{\ell_j(\tau)\}$, $j=1,\dots,k$,
defined over the nodes $\tau_i$, $i=1,\dots,k$: $$
P_j(\tau)=\sum_{r=1}^k P_j(\tau_r) \ell_r(\tau),
 \qquad j=1,\dots,s.$$

\no It follows that, for $i=1,\dots,k$ and $j=1,\dots,s$:
$$\int_0^{\tau_i} P_j(x) \mathrm{d}x = \sum_{r=1}^k P_j(\tau_r)
\int_0^{\tau_i} \ell_r(x) \mathrm{d}x = \sum_{r=1}^k P_j(\tau_r)
\alpha_{ir},$$

\no that is (see (\ref{OIP})--(\ref{IDPO}) and (\ref{collocation_rk})),
$\I_s = \mathcal A \P_s$. Consequently,
$${\mathcal A}\P_s\P_s^T\O = \I_s\P_s^T\O,$$
so that one retrieves the tableau (\ref{rk}) which defines the method
HBVM$(k,s)$.\QED

\medskip
The resulting Runge-Kutta method \eqref{hbvm_rk1} is then energy
conserving if applied to polynomial Hamiltonian systems
\eqref{hamilode}, when the degree of $H(y)$ is lower than or equal
to a quantity, say $\nu$, depending on $k$ and $s$. As an example,
when a Gaussian distribution of the nodes $\{\tau_i\}$ is
considered, one obtains (\ref{knu}) and, moreover, HBVM$(k,s)$ is
also related to the Gauss-Legendre method of order $2k$, according
to (\ref{hbvm_rk1}), whose Butcher array coincides with $\cal A$,
with this choice of the nodes $\{\tau_i\}$.

\begin{rem}It seems like the price paid to achieve such conservation property
consists in the lowering of the order of the new method with
respect to the original one \eqref{collocation_rk}. Actually this
is not true,  because a fair comparison would be to relate method
\eqref{rk}--\eqref{hbvm_rk1}  to a collocation method constructed
on $s$ rather than on $k$  stages, since the actual nonlinear
system, deriving by the implementation of HBVM$(k,s)$, turns out to have
dimension $s$, as it has been shown in \cite{BIT}.
\end{rem}

Further implications of the isospectral property of HBVMs, among
which an alternative proof for their order of convergence, may be
found in \cite{BIT3}. A further alternative proof can be found in
\cite{BITUOVO,BITframe}.

\section{Conclusions}\label{fine}
In this paper, we have shown that the recently introduced class of
energy-preserving methods
$\{$HBVM$(k,s)\}$, when recast as Runge-Kutta methods, have the matrix
of the corresponding Butcher tableau sharing the same nonzero
eigenvalues which, in turn, coincides with those of the matrix of
the Butcher tableau of the Gauss method of order $2s$, for all
$k\ge s$ such that $B(2s)$ holds.

Moreover, HBVM$(k,s)$ defined at the Gaussian nodes
$\{\tau_1,\dots,\tau_k\}$ on the interval $[0,1]$ are closely
related to the Gauss method of order $2k$ which, although symplectic, is not in
general energy-preserving.

\end{document}